 \documentclass[12pt,reqno]{amsart}
\usepackage{amssymb , amsmath, amsthm}
\usepackage{amsmath}
\usepackage{amsbsy}
\usepackage{amssymb}
\usepackage{color}
\usepackage[dvips]{graphicx}
\usepackage{subfigure}
\usepackage[active]{srcltx}

\usepackage{amsmath}
\usepackage{amsfonts}
\usepackage{amssymb}
\usepackage{amstext}
\usepackage{amsbsy}
\usepackage{epsf}

\newtheorem{theo}{Theorem}[section]
\newtheorem{prop}{Proposition}[section]
\newtheorem{coro}{Corollary}[section]
\newtheorem{lemma}{Lemma}[section]
\newtheorem{definition}{Definition}[section]

\newtheorem{remark}{Remark}[section]

\newcommand{\thmref}[1]{Theorem~\ref{#1}}

\def\r{\mathbb R}

\def\h{\mathbb H}

\def\n{\mathbb N}

\def\Hip{\mathbb H}

\let\hip=\Hip

\newcommand{\cal}{\mathcal}

\newcommand{\R}{\mathbb R}

\newcommand{\wt}{\widetilde}

\let\h=\hip

\let\HH=\hip

\def\Mc{\mathcal M}

\def\rmd{\mathop{\rm d\kern -1pt}\nolimits}
\def\rme{\mathop{\rm e\kern -1pt}\nolimits}

\def \Cc{{\mathcal C}}

\def \Mc{{\mathcal M}}

\def\bel{ \medskip
 \centerline{$ \ast \hbox to 1.0cm{}\ast \hbox to 1.0cm{}\ast $}
}

\def\longerrightarrow{-\kern-5pt\longrightarrow}

\def\star{\lower 1pt\hbox{*}}
\def \nulset {
\raise 1pt\hbox{ \hskip -3pt$\not$\kern -0.2pt \raise
.7pt\hbox{${\scriptstyle\bigcirc}$}}}

\newcommand{\hi}[1]{\mathbb{H}^#1}

\newcommand{\ch}{\cosh}

\newcommand{\pain}{\partial_{\infty}}

\newcommand{\ov}[1]{\overline{#1}}

\let\leq=\leqslant
\let\geq=\geqslant

\begin{document}

\title[Maximum Principle and symmetry]{Maximum
Principle and Symmetry for\\ Minimal Hypersurfaces
in ${\mathbb H}^n\times {\mathbb R}$ }

\author{Barbara Nelli, Ricardo Sa Earp, Eric Toubiana}

\address{Barbara Nelli {\em (corresponding author)}\newline
Dipartimento Ingegneria e Scienze dell'Informazione e Matematica \newline
Universit\'a di L'Aquila \newline 
via Vetoio - Loc. Coppito \newline
67010 (L'Aquila) \newline
 Italy}
\email{nelli@univaq.it}

 \address{Ricardo Sa Earp\newline
 Departamento de Matem\'atica \newline
  Pontif\'\i cia Universidade Cat\'olica do Rio de Janeiro\newline
Rio de Janeiro \newline
22453-900 RJ \newline
 Brazil }
\email{earp@mat.puc-rio.br}

\address{Eric Toubiana\newline
Universit\'e Paris Diderot - Paris 7 \newline
Institut de Math\'ematiques de Jussieu, UMR CNRS 7586 \newline
UFR de Math\'ematiques, Case 7012 \newline
B\^atiment Chevaleret \newline
75205 Paris Cedex 13 \newline
France}
\email{toubiana@math.jussieu.fr}

\date{}



 \thanks{ 
 The authors were partially supported by CNPq
 and FAPERJ of Brasil.}

\begin{abstract}
The aim of this work is to study how the asymptotic boundary   of a  
minimal hypersurface in $\h^n\times\r$
determines the behavior of the hypersurface at finite points, in several geometric situations.
\end{abstract}

\maketitle


\section{Introduction}
\label{introduction}

In this  article we  discuss how, in several geometric situations, the shape at infinity of a minimal surface in
$\h^2\times\r$ determines the shape of the surface itself. 

A beautiful theorem in minimal  surfaces theory is the Schoen's
characterization
of the catenoid \cite{Schoen}. It can be stated as follows. 
{\it  Let $M\subset \R^3$ be a complete immersed minimal surface with two annular ends.
Assume that each end is a graph, then $M$ is a catenoid.} 

On  the other hand, there exists  a complete  minimal
annulus immersed in a
slab of $\R^3$ \cite{R-T}.

A characterization of the catenoid in the hyperbolic space,
assuming 
regularity at infinity, was established by   G. Levitt and H. Rosenberg in 
\cite{Lev-Ros}. In a joint work with L. Hauswirth \cite{HNST},  the authors  of
the present article proved  a 
Schoen type theorem in $\hi2\times \r$, \ in the class of finite total
curvature surfaces.

\smallskip

Our first result is  a  new    Schoen type theorem in
$\hi2 \times \R$.   Namely,  we replace Schoen's  assumption {\em each end is a graph}
with 
the assumption {\em each end is a vertical graph whose  asymptotic boundary is a copy
of
the  asymptotic boundary of  $\hi2$} (Theorem \ref{T.Catenoid}).

\smallskip
Our second result is  a  {\em maximum principle}  in a vertical (closed) halfspace.  Assume
that $M$ is a complete
minimal
surface, possibly with finite boundary, properly immersed
in $\h^2\times\r$ and 
that  the boundary of $M$, if any, is contained in  the closure of a  
 vertical halfspace
$P_+.$ Assume further that the points at finite height of the asymptotic 
boundary of $M$ are contained in the asymptotic boundary  of the halfspace
$P_+.$ Then $M$ is entirely contained in the halfspace $P_+$, unless $M$ is equal to the
vertical halfplane $\partial P_+$ 
 (Theorem \ref{plano}).


\smallskip
 Then we generalize our results to higher dimensions. 

Theorem  \ref{T.Catenoid} and Theorem \ref{plano}  in higher dimension are analogous
to the 2-dimensional case. 
In order to generalize Theorem  \ref{T.Catenoid}, we first need  to  give a
characterization of the 
$n$-catenoid  analogous to that of  the 2-dimensional case 
(Theorem \ref{T.n-Catenoid},  see also \cite{B-SE}).

 Moreover in the higher dimensional case, it is worthwhile to state some interesting
consequences of our results.

Let $S_\infty$ be a closed set 
contained in an open slab of $\pain \h^n \times \r$ with height equal to $\pi/(n-1)$ 
such that the projection of $S_\infty$ on 
$\pain \h^n \times\{0\}$ omits an open subset.

We prove that there is no complete properly immersed 
minimal hypersurface $M$ whose  asymptotic boundary is  $S_\infty$
(Theorem \ref{T.slab.catenoid}-(\ref{item.non existence})).

Finally we prove an Asymptotic Theorem (Theorem \ref{T.Asymptotic Theorem}),
 that  implies the following non-existence
result.
There is no horizontal minimal
graph over a  bounded strictly convex domain, see
\cite[Equation (3)]{Sa2}, given by a positive function $g$ continuous 
up to the boundary,
taking zero boundary value data (Remark \ref{R.graph}).

\bigskip

\subsection*
{Acknowledgements} The first and the third authors wish to thank {\em
Departamento de
Matem\'atica da
PUC-Rio} for
the kind hospitality. 
The first and second authors
wish to thank {\em Laboratoire} {\em G\'eom\'etrie et Dynamique de l'Institut de
Math\'ematiques de Jussieu} for the kind hospitality.

\section{A characterization of the catenoid in $\hi2\times \R$}

We are going to prove  the  characterization of the catenoid presented in
the
Introduction.

Any  surface in $\h^2\times \r$ with constant third coordinate is a complete
totally  geodesic surface called a {\em slice}.
 For any  $s\in \R$, we
denote 
by $\Pi_s$ the slice $\hi2 \times \{s\}$ and we set 
$\Pi_s^+=\{(p,t) \mid p\in \hi2,\  t>s\}$ and 
$\Pi_s^-=\{(p,t) \mid p\in \hi2,\  t<s\}$.
 For simplicity $\Pi$ stands for 
$\Pi_0$.

\begin{lemma}\label{L.vertical}
 Let $\Gamma^+$ and $\Gamma^-$ be two Jordan curves in $\pain\hi2 \times \R$ 
which are vertical graphs over $\pain \hi2 \times \{0\}$ and such that 
$\Gamma^+\subset \pain \Pi^+$ and $\Gamma^-\subset \pain \Pi^-$. Assume that 
 $\Gamma^-$ is the  symmetry  of $\Gamma^+$ with respect to $\Pi$.

Let $M\subset \hi2 \times \R$ be an  immersed, connected,  complete minimal
surface 
with two ends $E^+$ and $E^-$. Assume that each end is a vertical graph
and that $\pain M= \Gamma^+ \cup \Gamma^-$, that is
$\pain E^+= \Gamma^+$ and $\pain E^-= \Gamma^-$.

Then $M$ is symmetric with respect to $\Pi$. Furthermore, each part  
$M\cap \Pi^\pm$ is a vertical graph and $M$ is embedded. 

\end{lemma}

\begin{proof}
 For any $t>0$ we set $M_t^+=M \cap \Pi_t^+$. We denote by 
$M_t^{+*}$ the symmetry of  $M_t^+$ with respect to the slice $\Pi_t$.
Furthermore, we denote by $t^+$ the highest $t$-coordinate of $\Gamma^+$.
Since $\pain M=\Gamma^+ \cup \Gamma^-,$ then
$M\cap \Pi_{t^+}=\emptyset$, by the maximum principle. 

 We denote by $E^+$  the end of $M$ whose asymptotic boundary is $\Gamma^+$. 
As $E^+$ is a vertical graph, there exists $\varepsilon >0$ such that  $M_{t^+ -\varepsilon}^+$ 
is a vertical graph, then  we can start Alexandrov reflection \cite{Alexandrov}.

\smallskip

 We keep doing the Alexandrov reflection with $\Pi_t$, doing $t\searrow 0$.  By applying the 
interior or boundary maximum principle, 
we get that, for
$t>0$,  the surface $M_t^{+*}$ stays above $M_t^-$. Therefore we get that $M_0^+$ is a vertical graph and
that 
$M_0^{+*}$ stays above $M_0^-$. 

Doing  Alexandrov reflection with slices coming from below, 
one has that 
$M_0^-$ is a vertical graph and that 
$M_0^{-*}$ stays below $M_0^+$, henceforth we get  
$ M_0^{+*}=M_0^{-}$.
Thus $M$ is symmetric with respect to $\Pi$ and each component of $M \setminus\Pi$ is
a graph. Therefore we can show, as in   the proof of 
\cite[Theorem 2]{Schoen},
that the whole surface $M$ is embedded. 
This completes the proof.
\end{proof}



\begin{definition}
 {\em A {\em vertical plane} is  a complete totally geodesic
surface 
$\gamma\times \R$ where $\gamma$ is any complete geodesic
of $\hi2$.}
\end{definition}

\begin{theo}\label{T.Catenoid}
 Let $M\subset \hi2 \times \R$ be an  immersed, connected, complete minimal
surface 
with two  ends. Assume that each end is a vertical graph whose asymptotic
boundary is a copy of $\pain \hi2$. 
Then $M$ is rotational, hence $M$ is a catenoid.
\end{theo}

\begin{proof}
Up to a vertical translation, we can assume that the asymptotic boundary is
symmetric
with respect to the slice $\Pi$. We use the same notations
as in the proof of Lemma \ref{L.vertical}. We know from  Lemma
\ref{L.vertical} that $M$ is symmetric with respect to $\Pi$ and that $M_0^+$
and $M_0^-$ are vertical graphs.  Therefore, at any point of $M\cap \Pi$ the
tangent plane of $M$ is orthogonal to $\Pi$.

We have $\pain M=\pain \hi2 \times \{t_0,-t_0\}$ for some $t_0>0$. Since $M$ is
embedded, $M$ separates $\hi2 \times [-t_0,t_0]$ into two connected components. 
We denote by $U_1$ the component whose  asymptotic boundary is 
$\pain \hi2 \times  [-t_0,t_0]$ and by $U_2$ the component such that 
$\pain U_2= \pain \hi2 \times \{t_0,-t_0\}$.

Let $q_\infty \in \pain \hi2 $ and let $\gamma \subset \hi2$ be an oriented 
geodesic issuing from $q_\infty$, that is $q_\infty \in \pain \gamma$. Let
$q_0\in \gamma$ be any fixed point.

For any $s\in \R$, we denote  by $P_s$ the vertical plane orthogonal
to $\gamma$ passing through the point of $\gamma$ whose  oriented distance
 from  $q_0$ is $s$. We suppose that $s<0$ for any point in the
geodesic segment 
$(q_0, q_\infty)$.

For any $s\in \R$, we call $M_s (l)$ the part of $M \setminus P_s$
such that $(q_\infty, t_0) , (q_\infty, -t_0) \in \pain M_s(l)$ and let 
 $M_s^* (l)$ be the reflection  of $M_s (l)$ about $P_s$. 
We denote by 
$M_s (r)$ the other part of $M\setminus P_s$ and by $M_s^* (r)$ its reflection 
about $P_s$.

\smallskip

  It will be clear from the following two Claims, 
why we can start the Alexandrov reflection principle with respect to 
the vertical planes $P_s$ and obtain the result.
\smallskip

By assumptions there exists $s_1<0$ such that for any $s<s_1$ the part 
$M_s(l)$ has two connected components and both of them are vertical graphs.
We deduce that $\partial M_s(l)$ has two (symmetric) connected components, each
one being a vertical graph. 

\smallskip

We recall that  $\Pi^+ :=\{t>0\}$ and $\Pi^- :=\{t<0\}$.

\smallskip 

\noindent {\em Claim 1. For any $s<s_1,$ we have that $M_s^* (l)\cap \Pi^+$
stays above $M_s(r)$ and $M_s^* (l)\cap \Pi^-$ stays
below $M_s(r)$. Consequently  $M_s^* (l)\subset U_2$ for any $s<s_1$. }

\smallskip

Observe that $M_s^* (l)\cap \Pi^+$ and $M_s(r)\cap \Pi^+$ have same
asymptotic boundary and that 
$\partial \,(M_s^* (l)\cap \Pi^+)= \partial M_s(r)\cap \Pi^+$. Therefore the
asymptotic and finite boundaries of any  lifting up of $M_s^*(l)$ is
above 
the asymptotic and finite boundaries of $M_s(r)$. Hence any   
lifting up of $M_s^*(l)$ is above $M_s(r)$ by the maximum principle, which
ensures
that the whole $M_s^* (l)\cap \Pi^+$ stays
above $M_s(r)$ for any $s<s_1,$ as desired. The proof of the other
assertion is analogous. Then, Claim 1 is  proved.

\smallskip 
 We set 
\begin{equation*}
\sigma =\sup\big\{s \in \R \mid  M_t^* (l)\cap \Pi^+\ 
\text{stays above}\ M_t(r)\cap \Pi^+\ \text{for any}\ 
t\in (-\infty, s)\big\}.
\end{equation*}

\noindent {\em Claim 2. We have $ M_\sigma^* (l)= M_\sigma(r)$. Thus,
given a geodesic $\gamma\subset \hi2$, there exists a vertical 
plane $P_\sigma$ orthogonal to $\gamma$ such that $M$ is 
symmetric with respect to $P_\sigma$ 
}

\smallskip

Note that we  also have 
\begin{equation*}
\sigma =\sup\big\{s \in \R \mid  M_t^* (l) \subset U_2
\ \text{for any}\ 
t\in (-\infty, s)\big\}.
\end{equation*}

 In order to prove Claim 2, we first establish the following fact.

\smallskip
\noindent {\em Assertion. For any $s$ such that  
$  M_s^* (l)\cap \Pi \subset U_2$ then 
 $  M_s^* (l) \subset U_2$.}

\smallskip

As $M$ is symmetric with respect to $\Pi$ the intersection $M\cap \Pi$ is
constituted of a finite number of pairwise disjoint Jordan curves
$C_1,\dots,C_k$. Since $M\cap \Pi^+$ is a vertical graph we deduce 
\begin{equation*}
 (C_j \times \R) \cap M= C_j \quad \text{for any}\ j=1,\dots,k.
\end{equation*}
Moreover, since $M$ is connected and is symmetric about $\Pi$, we get that 
$M\cap \Pi^+$ is connected.

Let $D_j\subset \Pi$ be the Jordan domain bounded by $C_j$, $j=1,\dots,k$.
Noticing that:
\begin{itemize}
\item $(M\cap \Pi^+)\setminus (\ov D_j \times \R)\not=\emptyset$,

\item  $M\cap \Pi^+$ is connected,

\item $M\cap (C_j \times \R)=C_j$,

\item $\pain M \cap \Pi^+ =\pain \hi2 \times \{t_0\}$,
\end{itemize}
we get that 
$(M\cap \Pi^+)\cap ( D_j \times \R)=\emptyset$, $j=1,\dots,k$. Hence, 
$D_i\cap D_j=\emptyset $ for any  $i\not= j$. Therefore, $M\cap \Pi^+$  is a
vertical graph over $\Pi \setminus \cup D_i$.  

This implies that, for any 
$\varepsilon >0,$ the vertical translation  $(M_s^*(l) \cap \Pi^+) +(0,0,\varepsilon)$ 
stays above  $M$. This proves the Assertion.

\smallskip

Let us continue the proof of Claim 2.
The definition of $\sigma$ implies that 
$M_{\sigma + \varepsilon} ^* (l) \cap U_1
\not=\emptyset,$ for $\varepsilon$ small enough.  

We
deduce from the Assertion that $M_{\sigma +\varepsilon}^* (l)\cap \Pi $ is not
contained in $U_2$ for any small enough $\varepsilon >0$.  Hence we infer that 
$M_\sigma^* (l)\cap \Pi$ and $M_\sigma (r)\cap \Pi$ are tangent at an interior or boundary
point lying in some Jordan curve $C_j$ contained in $M\cap \Pi$.  
Since $M_\sigma^* (l) \subset \ov U_2$,  
$M_\sigma (r) \subset  \partial U_2$ and the tangent plane of $M$ is vertical along
$M\cap \Pi$, 
 we are able to apply the
maximum principle (possibly with boundary) to conclude that 
$ M_\sigma^* (l)= M_\sigma(r)$, that is $P_\sigma$ is a plane of symmetry 
of $M$. This proves Claim 2.

\smallskip

For any $\alpha \in (0,\pi/2]$  consider a family of vertical  planes
making an angle $\alpha$ with $P_\sigma$, generated
by
hyperbolic translations along the horizontal geodesic $P_\sigma \cap \Pi$.
Now, doing the Alexandrov  reflection principle with this family of planes, we
find a vertical plane  of symmetry of $M$, say $P^\alpha.$  Hence $M$ is invariant by the rotation of angle 
$2\alpha$ around the
vertical geodesic $P^\alpha \cap P_\sigma$.  Choosing an angle $\alpha$ such
that $\pi /\alpha$ is not rational, we find that $M$ is invariant by rotation
around the axis $P^\alpha \cap P_\sigma$. This concludes the proof of 
Theorem \ref{T.Catenoid}, as desired.
\end{proof}

\begin{remark}
{\em  For any integer $n,$ there exists a minimal surface in $\hi2 \times \R$ which is
a
vertical graph, whose  asymptotic boundary is a copy of $\pain \hi2$ 
and whose  finite boundary is constituted of $n$ smooth Jordan curves in the
slice $\Pi$, see \cite[Theorem 5.1]{ST2}.  In the same article the  second and third
author 
asked  about the existence of such graphs with two boundary curves 
in $\Pi$ cutting orthogonally the slice $\Pi$. Theorem \ref{T.Catenoid} implies 
that the answer to this question is negative.}
\end{remark}

\section{  Maximum Principle  in a vertical halfspace of $\h^2\times\r.$}
\label{dimension-two}

In this section we prove  some maximum principle  in a vertical halfspace. 
More precisely, we prove that, under some geometric assumptions, the behavior of the asymptotic boundary of $M$ at finite 
height, determines the behaviour of $M$.

\smallskip

\begin{definition}\label{D.vertical-halfspace}
{\em We call a
{\em vertical  halfspace}
any of the two components of 
$(\hi2 \times \R)\setminus P$, where $P$ is a vertical plane.}
 \end{definition}

\begin{theo}\label{plano}
Let $M$ be a complete minimal surface, possibly with  finite boundary, properly immersed in
$\h^2\times\r.$
Let $P$ be a vertical  plane and let $P_+$ be one of the two  
 halfspaces determined by $P.$ If $\partial M \subset \overline{P_+}$ and $
\partial_{\infty}M\cap(\partial_{\infty}\hi2\times\r)\subset\partial_{\infty} P_+,$  then
$M\setminus \partial M \subset P_+,$ unless $M\subset P.$
\end{theo}

 For the proof of Theorem \ref{plano} we need to consider the one parameter 
family of surfaces $M_d, d>0,$ that have origin in \cite[Section 4]{Sa} and whose geometry
is
described in \cite[Proposition 2.1]{ST2}.   This family of surfaces  was already used,
for example, in \cite[Example 2.1]{Sa1}.

 First we describe the asymptotic boundary of $M_d$, for $d>1$.

Consider a horizontal geodesic $\gamma$ in $\h^2$, with asymptotic 
boundary $\{p,q\}$ and let $\alpha$ be the closure of a connected component 
of $(\pain\hi2 \times \{0\})\setminus (\{p,q\}\times \{0\})$.
Let
$$
H(d)=\int_{\cosh^{-1}(d)}^{+\infty}\frac{d}{\sqrt{\cosh^2 u-d^2}}\,du, \ \ d>1
$$
be the positive number defined in (1) of \cite{ST2}.
 Notice that ${\displaystyle\lim_{d\longrightarrow 1}H(d)=+\infty}$ and 
${\displaystyle\lim_{d\longrightarrow +\infty}H(d)=\pi/2}$.

 Let   $\alpha_d$  in   $\partial_{\infty}\h^2\times\{H(d)\}$  and $\alpha_{-d}$
in    $\partial_{\infty}\h^2\times\{-H(d)\}$  be the two curves that  project
vertically onto  $\alpha.$  Let  $L_d,$ $R_d$ be two vertical segments in
$\partial_{\infty}\h^2\times\r$  of height $2H(d)$ such that the curve
$L_d\cup\alpha_d\cup R_d\cup\alpha_{-d}$ is a closed simple  curve. Then 
$\partial_{\infty}M_d=L_d\cup\alpha_d\cup R_d\cup\alpha_{-d}$.

Now we describe the position of $M_d$ in the ambient space, for   $d>1$.

Denote by $Q_{\gamma}$  the  halfspace determined by
$\gamma\times\r,$ whose asymptotic boundary contains the curve
$\alpha.$ Let $\gamma_d$ be the curve in
$Q_{\gamma}\cap(\h^2\times\{0\})$ at  constant distance
$\cosh^{-1}(d)$ from $\gamma.$ $M_d$ contains the curve $\gamma_d.$
Denote by $Z_d$ the  closure of the non mean convex side of the cylinder
over the curve $\gamma_d $. Then, 
 $M_d$ is contained in  $Z_d$  which  is contained in $
Q_{\gamma}$. Notice that  any  vertical
translation of the surface $M_d$ is contained in $Z_d.$
 Moreover, any vertical translation of
$M_d$ is arbitrarily close to $Q_\gamma$ if $d$ is sufficiently close to 1.

 We observe that in the description above, 
$\gamma$ can be any geodesic of $\h^2.$

\smallskip

{\it Proof of Theorem \ref{plano}.}  The proof is an application of the maximum 
principle between the surface $M$ and the  one parameter family of surfaces
$M_d.$ 

\smallskip

We choose  the geodesic $\gamma,$ in order to construct the  $M_d$'s, as follows.
Let $\gamma \subset \hi2$ be any geodesic such that 
\begin{itemize}
\item{P1}: \label{item.geod1}The halfspace $Q_{\gamma}$ is strictly contained in
$(\h^2\times \r)\setminus P_+$.

\item{P2}: \label{item.geod2} $\pain\gamma\cap \pain P=\emptyset$.
\end{itemize}



Now, notice that

\begin{enumerate}
\item \label{item.asymptotic}The intersection of $\partial_{\infty}M$ with 
 $\partial_{\infty}(\h^2\times\r)\setminus \partial_{\infty} P_+$ 
  contains no points at finite
height. 
\item The asymptotic boundary of any vertical translation $M_d$ is  contained in the
asymptotic
boundary of  $ Q_{\gamma}\subset\h^2\times\r\setminus  P_+.$

\end{enumerate}

 We claim that $M_d$ and $M$ are disjoint for any $d>1$. Indeed,
 letting $p\longrightarrow q$
 (recall that $p,$ $q$ are the endpoints  of the geodesic $\gamma$), one has 
that $M_d$ collapses to a vertical segment in $\partial_{\infty}\h^2\times\r.$
Suppose that, when $p\longrightarrow q$, the surfaces $M_d$ always have a nonempty
intersection with $M$. Then, there would exists a point of 
the asymptotic boundary of $M$ at finite height in
 $\partial_{\infty}(\h^2\times\r)\setminus \partial_{\infty}P_+$, giving a contradiction with
\eqref{item.asymptotic}.

Then, if $M\cap M_d\not=\emptyset$, we would obtain a last
intersection point between $M$ and some modified $M_d$ letting
$p\longrightarrow q$,  contradicting the maximum principle.

\smallskip

Therefore, by the maximum principle, any vertical
translation of $M_d$ and $M$ are disjoint.

\smallskip

Let $d\longrightarrow  1.$ By the maximum principle, there is no
first  point of contact between  $M_d$ and $M.$  As we can apply the
maximum principle  between any vertical translation of $M_d$ and
$M,$ one has that   $M$  is contained in the closed halfspace
$\h^2\times\r\setminus Q_{\gamma}$ for any geodesic $\gamma$ satisfying 
the properties P1 and P2. Therefore, $M$ is included in the closure of 
$P_+$.

\smallskip

Now we have one of the following possibilities:
\begin{itemize}
\item  Some points of the interior of $M$ touches  $\partial P_+=P,$ then, by the maximum principle, 
$M\subset P.$
\item $M\setminus \partial M$ is contained in the halfspace $P_+.$
\end{itemize}

The result is thus proved. \qed

\smallskip

Let us give a definition, before stating some  consequences of Theorem  \ref{plano}. 


\begin{definition}\label{D.convex-envelope}
{\em We say that $L\subset \partial_{\infty}(\h^2\times\r)$ is a {\em line} if $L=\{p\}\times\r$ for some $p\in\partial_{\infty}\h^2.$

Given  vertical lines  $L_1,\dots,L_k$ in $\partial_{\infty}\h^2\times\r,$
we  define the set 
  $P(L_1,\dots,L_k)$ as follows.
Let $P_i$ the vertical plane such that 
$\partial_{\infty}P_i\cap (\partial_{\infty}\h^2\times\r)=L_i\cup L_{i+1}$ (with
the convention that $L_{k+1}=L_1$).
Denote
by $\tilde P_i$ the  halfspace determined by the vertical  plane $P_i$ such
that  $\bigcup_{j}L_j\subset \partial_{\infty}\tilde P_i.$ Then,
we set $P(L_1,\dots,L_k):=\cap_i\tilde P_i$.}
\end{definition}

\begin{coro}
\label{coro-k-lines}
Let $M$ be  a complete minimal surface, possibly with finite boundary, properly immersed
in
$\h^2\times\r$ and let
$\Gamma=\partial_{\infty}M\cap (\partial_{\infty}\h^2\times\r).$ Let $L_1,\dots,
L_k$ be vertical lines in $ \partial_{\infty}\h^2\times\r.$ If $\Gamma\subset
L_1\cup\dots\cup L_k$ and $\partial M\subset \overline{ P(L_1,\dots, L_k)},$ then $M\setminus\partial M$ is
contained in
$P(L_1,\dots, L_k)$, 
 unless $M$  is contained in  one of the $P_i.$
\end{coro}

\begin{proof}
By Theorem \ref{plano}, $M$ is contained in every halfspace  $\tilde P_i$ 
determined by the vertical plane
$P_i$ such that  $\bigcup_{j}L_j\subset \partial_{\infty}\tilde P_i,$ unless  it
 is contained in  one of the $P_i.$  Hence it is contained in $P(L_1,\dots, L_k),$
by definition,
unless it is contained in  one of the $P_i.$
\end{proof}

\begin{coro}
\label{coro-halfspace}
Let $M$ be a complete minimal  surface properly  immersed in  $\h^2\times\r.$
Let $P$ be a vertical  plane. If
$\partial_{\infty}M\cap( \partial_{\infty}\h^2\times\r)\subset\partial_{\infty}P,$ then
$M=P.$
\end{coro}

\begin{proof}
By Theorem \ref{plano}, $M$ is contained in the closure of both halfspaces
determined  by $P,$
hence it is contained in $P.$ Then $M= P$ because it is complete.
\end{proof}

\begin{coro}
\label{one-end}  Let $M$ be a complete minimal  surface properly   immersed in
$\h^2\times\r$. Suppose  that  the asymptotic boundary of $M$  is  contained  in the asymptotic
boundary of a totally geodesic plane $S$ of $\h^2\times\r$. 
Then $M=S$.
\end{coro}

\begin{proof} The proof is a simple consequence of the maximum principle and of the previous results. We do it for completeness.
 First assume that    the asymptotic boundary  of $M$ is contained in
the asymptotic boundary of
a  slice, say $\{t=0\}.$ Then, for $n$ sufficiently large, the slice  $\{t=n\}$ 
is 
 disjoint from $M.$ Now,   we translate the slice  $\{t=n\} $ down. The first 
contact point, cannot be interior because of the maximum principle, hence $M$
must stay below 
 the slice $\{t=0\}.$ One can do the same reasoning with slices coming  from the
bottom, and $M$ must stay above  the slice $\{t=0\}.$ Hence $M$
 coincides with  the slice $\{t=0\}.$ 
 
  If  the  the asymptotic boundary of  $M$ is contained is  
the asymptotic boundary of a
vertical plane, the result follows by Corollary \ref{coro-halfspace}. 
 \end{proof}

\begin{coro}
\label{coro-boundary}
Let $M$ be a complete minimal surface properly immersed in $\h^2\times\r.$
Assume that the projection of the asymptotic boundary of  $M$ into
$\partial_\infty \h^2 $ omits a
closed interval $\alpha$ joining two points $p$ and $q$. Let  $\gamma$ be the
horizontal geodesic  in $\h^2$
whose the asymptotic boundary is  $\{p,q\}$ and let $Q_\gamma$
be  the  halfspace determined by $\gamma \times\r$ whose asymptotic boundary contains
$\alpha.$ Then $M$ is contained in $\h^2\times\r \setminus \overline Q_\gamma.$
\end{coro}

\begin{proof}
By hypothesis $\partial_{\infty}M\cap(\partial_{\infty}\h^2\times\r)$ is
contained  
in the asymptotic boundary of $(\h^2 \times \r)\setminus Q_{\gamma}.$ The result follows
by 
Theorem \ref{plano} 
with $P_+=(\h^2\times\r)\setminus\overline Q_{\gamma}.$
\end{proof}

\begin{remark} 
{\em There exist examples of minimal surfaces with asymptotic boundary
equal  to two vertical halflines, lines  and a curve at finite height,  
see \cite[Equation (32)]{Sa} and \cite[Proposition  2.1 (2)]{ST2}.}
\end{remark}

\smallskip

\section{ Some generalizations to $\hip^ n\times\R.$}
\label{dimension-n}

Let us recall the construction and the properties  of the $n$-catenoids
in $\hi n\times \R$, $n\geq 3$, established,  by P. B\'erard and the
second author in 
\cite[Proposition 3.2]{B-SE}. 
Given any $a>0$ we denote by 
 $\big(I_a, f(a,\cdot) \big)$, where $I_a\subset\r$ is an interval,  the maximal solution
of the following Cauchy problem:
\begin{equation*}
\begin{cases}
 f_{tt}=(n-1)(1+f_t^2)\,\text{coth}(f), \\
f(0)=a>0, \\
f_t(0)=0.
\end{cases}
\end{equation*}

\begin{theo}[\cite{B-SE}]\label{T-cat3-10a}
For $a > 0$, the maximal solution $\big(I_a, f(a,\cdot) \big)$
gives rise to the generating curve $C_a$,  parametrized by $t \mapsto \big(
\tanh(f(a,t)),t\big)$, of a complete minimal
rotational  hypersurface $\Cc_a$ $($\emph{n-catenoid}$)$ in 
$\Hip^n \times\R$, with the following properties.
\begin{enumerate}
    \item The interval $I_a$ is of the form $I_a = ]-T(a), T(a)[$
where 
\begin{equation*}
 T(a)=\sinh^{n-1}(a) \int_a^{\infty} \big( \sinh^{2n-2}(u) -
    \sinh^{2n-2}(a)\big)^{-1/2} \, du.
\end{equation*}

\item  $f(a,\cdot)$ is an even function
    of the second variable.
    \item For all $t \in I_a$, $f(a,t) \ge a$.
    \item The derivative $f_t(a,\cdot)$ is positive on $]0,T(a)[$,
    negative on $]-T(a),0[$.
    \item The function $f(a,\cdot)$ is a bijection from $[0,T(a)[$
    onto $[a,\infty[$, with inverse function $\lambda(a,\cdot)$ given by
    \begin{equation*}
    \lambda(a,\rho) = \sinh^{n-1}(a) \int_a^{\rho} \big( \sinh^{2n-2}(u) -
    \sinh^{2n-2}(a)\big)^{-1/2} \, du.
    \end{equation*}
    \item The catenoid $\Cc_a$ has finite vertical height $h_R(a):=2T(a)$,
    \item The function $a \mapsto h_R(a)$ increases from $0$ to
    $\frac{\pi}{(n-1)}$ when $a$ increases from $0$ to infinity.
    Furthermore, given $a \not = b$, the generating catenaries $C_a$ and $C_b$
    intersect at exactly two symmetric points.
\end{enumerate}
\end{theo}

For later use, we need the following result. Altough we believe that the result is
classical, we  give a proof for the sake of completeness.
The reader is referred to \cite[chapter VII]{Hopf} or 
\cite[chapter 9, addendum 3]{Spivak}
for the proof of the analogous statement  in Euclidean space.

\begin{prop}\label{P.classical}
  Let $S\subset \hi n$ be a finite union of connected,
closed and embedded 
$(n-1)$-submanifolds $C_j$, $j=1,\dots,k$, such that the bounded domains whose
boundary are the $C_j$ are pairwise disjoint.
 Assume that for  any geodesic
$\gamma \subset \hi n$, there
exists a $(n-1)$-geodesic plane  $\pi_\gamma \subset \hi n$ of
symmetry
of $S$ which is orthogonal to $\gamma$. Then 
$S$ is a $(n-1)$-geodesic sphere of $\hi n$.
\end{prop}

\begin{proof}
 We will proceed the proof by induction on $n\geq 2$.

First assume that $n=2$. By hypothesis, there exist two
geodesics $c_1, c_2 \subset \hi2$ of
symmetry of the closed curve $S$ intersecting at some point
$p\in \hi2$ and
making an angle $\alpha \not=0$ such that $\pi/\alpha$ is
not rational.
For any $q\in S$, denote by $C_q$ the
circle centered at $p$ passing  through $q$. Then $C_q$ is contained in $S$.
Let $\wt q\not= q$ be points of $S$. If $C_q\not= C_{\wt q}$ 
then the geodesic disks
bounded by $C_q$ and $C_{\wt q}$ 
are not disjoint, since they  have the same center, which contradicts the 
hypothesis.
Consequently, we get $C_q = C_{\wt q}$ and we  
conclude that $S$ is a circle.

\vskip4mm

Let $n \in \n$,  $n \geq 3$. Assume that the statement holds for  
$k=2,\dots,n-1$.

Let $\pi_0\subset \hi n$ be 
 a $(n-1)$-geodesic plane of symmetry of $S$.

\smallskip

\noindent {\it Claim 1. $S\cap \pi_0$ is a $(n-2)$-geodesic sphere of $\pi_0$}.

\smallskip

Indeed, let $\gamma \subset \pi_0$ be a geodesic. By hypothesis 
there exists a 
$(n-1)$-geodesic plane $\pi_\gamma \subset \hi n$ orthogonal to $\gamma$  which
is a plane of symmetry of $S$. Since $\pi_\gamma$ is orthogonal to $\pi_0$,
then $S\cap \pi_0$ is symmetric about $\pi_\gamma \cap \pi_0$ (which is 
a $(n-2)$-geodesic plane of $\pi_0$),  see \cite[Lemme 3.3.15]{ST3}.
As $\pi_0$ is
a $(n-1)$ hyperbolic space,  $S\cap \pi_0$ satisfies the assumptions of 
the statement in $\hi{{n-1}}$.

By the induction hypothesis we deduce that 
$S\cap \pi_0$ is a $(n-2)$-geodesic sphere of $\pi_0$. This proves Claim 1.

Let $p_0\in \pi_0$  and $\rho_0 >0$ be  respectively  the center and the radius
of the $(n-2)$-geodesic sphere $S\cap \pi_0$.

 \smallskip

\noindent {\it Claim 2. Let $\pi_1 \subset \hi n$ be  a 
$(n-1)$-geodesic plane of symmetry of $S$  orthogonal to $\pi_0$. 
Then $S\cap \pi_1$ is a $(n-2)$-geodesic sphere of $\pi_1$ with center 
$p_0$ and radius $\rho_0$.}

 \smallskip
Claim 1 yields that $S\cap \pi_1$ is a
$(n-2)$-geodesic sphere of
$\pi_1$. Since $\pi_0$ and $\pi_1$ are orthogonal, then the
geodesic sphere 
$S \cap \pi_0$ is symmetric about $\pi_1$. Therefore  $p_0 \in
\pi_1$.

If $n >3,$ then  $( S \cap \pi_0)\cap \pi_1$ is 
$(n-3)$-geodesic sphere with center $p_0$ and radius $\rho_0$  of $\pi_0 \cap
\pi_1$ (which is a $(n-2)$ hyperbolic space). If $n=3,$ then $( S \cap \pi_0)\cap
\pi_1$ is constituted of two points whose the distance is $2\rho_0$. In both 
cases we infer that 
$\text{diam}_{\h^n}\, (S\cap \pi_1)\geq 2\rho_0$ and then the radius of the geodesic
sphere 
$S\cap \pi_1$ is $\rho_1\geq \rho_0$.  Analogously we can show
that $\rho_0 \geq \rho_1$.  
We deduce  that $\rho_1=\rho_0$, that is $S\cap \pi_0$ and $S\cap \pi_1$
have both center at $p_0$ and radius $\rho_0$. This proves  Claim 2.

 \smallskip

\noindent {\it Claim 3. Let $\pi_2\subset \hi n$ be any 
$(n-1)$-geodesic plane of symmetry of $S$. Then $S\cap \pi_2$ is a 
 $(n-2)$-geodesic sphere of $\pi_2$ with center 
$p_0$ and radius $\rho_0$.}

 \smallskip

Since $S$ is symmetric with respect to $\pi_0$ and $\pi_2$, $\pi_0$ and $\pi_2$ are
distinct and $S$ is compact,
 then  the $(n-1)$-geodesic planes $\pi_0$ and $\pi_2$ cannot be
disjoint.

Then, we find  a third $(n-1)$-geodesic plane $\pi_3$ of symmetry of $S$, 
orthogonal to both $\pi_0$ and $\pi_2$. 
Claim 2 implies that $S\cap \pi_2$ is a $(n-2)$-geodesic sphere of
$\pi_2$ with center 
$p_0$ and radius $\rho_0$. This proves Claim 3.


\smallskip

Now we finish the proof of the Proposition as follows. Let $p\in S$  and let 
$\pi \subset \hi n$ be any $(n-1)$-geodesic plane passing through $p$ and 
$p_0$. Let $\gamma\subset \hi n$ be the geodesic through $p_0$ orthogonal 
to $\pi$. By Claim 2, there exists a $(n-1)$-geodesic plane $\pi_\gamma$ of
symmetry of $S$ and orthogonal  to $\gamma$. Claim 3 ensures that 
$p_0 \in \pi_\gamma$, then $ \pi_\gamma=\pi$.  Claim 3 yields also
that 
$S\cap \pi$ is  $(n-2)$-geodesic sphere of $\pi$ with center 
$p_0$ and radius $\rho_0$, thus $d_{\h^n}(p,p_0)=\rho_0$. 
This shows that $S$ is the  $(n-1)$-geodesic sphere of 
$\hi n$ of radius $\rho_0$ and center $p_0$.
\end{proof}

Now we establish a characterization of the $n$-catenoid, that is a generalization to
higher dimension of Theorem \ref{T.Catenoid}.

\begin{theo}\label{T.n-Catenoid}
 Let $M\subset \hi n \times \R$ be an  immersed, connected, complete minimal
hypersurface 
with two ends. Assume that each end is a vertical graph whose asymptotic
boundary is a copy of $\pain \hi n$. 
Then $M$ is a $n$-catenoid.
\end{theo}

\begin{proof}
Up to a vertical translation, we can assume that the asymptotic boundary of $M$  is
symmetric with respect to 
$\Pi:= \h^n \times \{0\}$. 
We set $\Gamma^+ :=\pain M \cap \{t>0\}$ and   recall that $\Gamma^+$
is a copy of $\pain \hi n$. 
As usual we set $M^+ := M \cap \{ t>0\}$.

\smallskip

Next  Claim can be shown in the same fashion as in  $\hi2 \times \R$ (see 
Lemma \ref{L.vertical} and the proof of  Claim 2 of   Theorem \ref{T.Catenoid}).
For
this reason we just state it. 

\smallskip

\noindent {\it Claim. M is symmetric about $\Pi$, 
and each connected component of 
$M\setminus \Pi$ is a vertical graph. 
 Moreover, for any geodesic $\gamma \subset\Pi$ there exists a vertical
hyperplane  $P_\gamma \subset\hi n \times \R$ orthogonal to $\gamma$ which is a
$n$-plane of symmetry of $M$. Therefore,  $\pi_\gamma:=P_\gamma \cap \Pi$
 is a $(n-1)$-plane of
symmetry
of $\Sigma:=M\cap \Pi $.}




 \smallskip

Using the result of the Claim we get that $\Sigma$ satisfies the assumptions 
of Proposition \ref{P.classical}. 
Then $\Sigma$ is a $(n-1)$-geodesic sphere of $\Pi$, since 
$\Pi=\hi n \times \{0\}$.

Let ${\cal C}\subset \hi n \times \R$ be the catenoid through $\Sigma$ and 
orthogonal to $\Pi$.  We set  
$ {\cal C}^+ :={\cal C} \cap \{ t>~0\} $.

Both $ {\cal C}^+$ and $M^+$  are vertical along their common 
finite boundary $\Sigma$, hence they are tangent along $\Sigma$.

Let $t_ {\cal C}$ (resp. $t_M$) the height of the asymptotic boundary 
of $ {\cal C}^+$ (resp. $M^+$).

Suppose for example that $t_ {\cal C} \leq t_M$. Then, 
 lifting upward and downward $M^+$, we obtain that  $M^+$ is above 
$ {\cal C}^+$. Therefore we deduce that $M^+={\cal C}^+$ by applying the boundary
maximum principle.
 The case  $t_M \leq t_ {\cal C}$ is analogous.

 We conclude that $M= {\cal C}$ and the proof is completed. 
\end{proof}

\bigskip

 In order to establish the generalization in higher dimension of Theorem \ref{plano},
 we need to state some existence results,
established for $n\geq 3$, 
in \cite[Theorem 3.8]{B-SE}, inspired by  
\cite[Proposition 2.1]{ST2}. 
 In fact, we only need the $d>1$ case, but we state the whole result for the sake of
completeness.

\begin{theo}[\cite{B-SE}]\label{T}
There exists a one parameter family $\{ \Mc_d,\, d>0\}$ of
complete embedded minimal hypersurfaces in $\HH^n \times \R$
invariant under hyperbolic translations.

\begin{enumerate}
    \item If $d>1$, then $\Mc_d$ consists of the union of two symmetric vertical
graphs over the exterior of an equidistant hypersurface in the
slice $\HH^n\times \{0\}$.

The asymptotic boundary of $\Mc_d$  is topologically an
$(n-1)$-sphere which is homologically trivial in $\partial_{\infty}
\HH^n \times \R$. More precisely, we set for $d>1$:
\begin{equation*}
 S(d)=\ch (a) \int_1^\infty (t^{2n-2}-1)^{-1/2} (\ch^2(a)t^2-1)^{-1/2}\,dt,\quad
\text{where } d =: \cosh^{n-1}(a).
\end{equation*}
Then, 
the asymptotic boundary of $\Mc_d$ consists
of the union of two copies of an hemisphere $S_+^{n-1}\times
\{0\}$ of $\partial_\infty \HH^n \times \{0\}$ in parallel slices
$t=\pm S(d)$, glued with the finite cylinder 
$\partial S_+^{n-1}\times [-S(d),S(d)]$

 The
vertical height of $\Mc_d$ is $2S(d)$. The height of the family
$\Mc_d$ is a decreasing function  of $d$ and varies from infinity
$($when $d\rightarrow 1)$ to $\pi /(n-1)$ $($when $d\rightarrow
\infty)$.

\item If $d=1$, then $\Mc_1$ consists of a complete  $($non-entire$)$
    vertical graph
over  a halfspace in $\HH^n\times \{0\}$, bounded by a totally
geodesic hyperplane $P$. It takes infinite  boundary value data on
$P$ and constant asymptotic boundary  value data. The asymptotic
boundary of $\Mc_1$ is the union of a spherical cap $S$ of
$\partial_\infty \HH^n\times \{c\}$ with a half vertical cylinder
over $\partial S$.
    \item If $d<1$, then $\Mc_d$ is an entire vertical graph with finite
vertical height. Its asymptotic boundary consists of a
homologically non-trivial $(n-1)$-sphere in $\partial_\infty \HH^n
\times \R.$
\end{enumerate}
\end{theo}\bigskip

The hypersurfaces ${\Mc}_d$ are the analogous in higher  dimension of the
surfaces $M_d$ in $\hi2\times\r$. Also, as in $\hi2\times\r$, 
by {\em $($vertical$\,)$ hyperplane} we mean  a complete totally geodesic
hypersurface 
$\Pi\times \R$, where $\Pi$ is any totally geodesic hyperplane
of $\h^n \times \{0\}$. Moreover, we call a
{\em vertical  halfspace} any component of 
$(\h^n \times \R)\setminus P$ where $P$ is a vertical hyperplane. Thus, working 
with the  hypersurfaces ${\Mc}_d$
exactly in the same way as in  \thmref{plano},  we
obtain the following result.

\begin{theo}
\label{iperplano}
Let $M$ be a complete minimal hypersurface properly immersed in
$\h^n\times\r$, possibly with finite boundary. Let $P$ be a vertical geodesic hyperplane
and
$P_+$ one of the two halfspaces determined by $P.$ If $\partial M \subset\overline{ P_+}$ and $
\partial_{\infty}M\cap(\partial_{\infty}\h^n\times\r)\subset\partial_{\infty} P_+,$  then
$M\setminus\partial M\subset P_+,$ unless $M\subset P.$
\end{theo}


Obviously, the analogous in higher dimension of Corollaries  \ref{coro-k-lines},
\ref{coro-halfspace}, \ref{one-end}  hold as well.

 Part (1) of next Theorem  is a  generalization in higher dimension of Corollary \ref{coro-boundary}, while part (2) was proved, for 
 $n=2$ by  the second and the third authors \cite[Corollary 2.2]{ST2}

\begin{theo}\label{T.slab.catenoid}
{Let $S_\infty \subset \pain \hi n \times \R$ be a closed set 
whose the vertical projection  on $\pain \hi n \times \{0\}$ omits
an open subset $U$.}
\begin{enumerate}
\item {Let  $M$ be a complete minimal hypersurface  properly
immersed in $\h^n\times\r$ such that $\pain M=S_\infty$.
Let $Q \subset \hi n \times \R$
be  a vertical  halfspace   whose asymptotic boundary is
contained in
$U \times \R.$ Then $M$ is contained in 
$\h^n\times\r \setminus  \ov Q.$}

\item\label{item.non existence} Assume that $S_\infty $ is contained in an open slab whose 
height is equal to $\frac{\pi}{n-1}$. Then, there is no complete connected  properly
immersed minimal
hypersurface 
$M$ in $\hi n \times \R$  with asymptotic boundary $S_\infty$.
\end{enumerate}
\end{theo}

\begin{proof}
 The first statement is a consequence of Theorem \ref{iperplano} and the 
proof is analogous    to that of
Corollary \ref{coro-boundary}.

\smallskip

Let us prove the second statement. Assume, by contradiction,
that there is such a minimal
hypersurface  $M$ with asymptotic boundary $S_{\infty}$. Then, up to a vertical
translation, we
can assume that $M$ is
contained in the slab  
$\mathcal{S}:=\{ \varepsilon<t<\frac{\pi}{n-1}-\varepsilon\}$ for some $\varepsilon>0$,
and thus $S_\infty\subset \pain \mathcal{S}$.
 By assumption, there exists  a $(n-1)$-geodesic plane 
$\pi\subset \hi n\times \{0\} $ such that a component $\pi^+$  of 
$\hi n \times \{0\} \setminus \pi$ satisfies:
\begin{enumerate}
\item $\pain \pi^+ \subset U$.
\item $M \cap (\pi^+ \times \R) =\emptyset$.
\end{enumerate}
Let $C\subset \hi n \times (0,\frac{\pi}{n-1})$ be
any $n$-catenoid such that a component of its asymptotic boundary stays stricly above 
$\pain\mathcal{S}$ and the other component stays strictly below $\pain\mathcal{S}$.

We take a connected and compact piece $K$ of $C$ such that its
boundary lies in the boundary of the slab $\mathcal{S}$. 

Let $q\in M$ be a point and let $q_0\in \hi n \times \{0\}$ be the vertical
projection of $q$.
Let $p_\infty \in \pain\pi^+$ be an asymptotic point. 
 Denote by $\wt \gamma\subset  \pain \hi n \times \{0\}  $
the complete geodesic passing through $q_0$ such that 
$p_\infty \in \pain \wt \gamma$. 
 We can translate $K$ along $\wt\gamma$ such that the translated $K$  is contained in
the halfspace $\pi^+ \times \R$.

Now we  come back translating $K$ towards $M$ along $\wt\gamma.$ 
Observe that the boundary of the translated
copies of $K$  does not touch $M$. Therefore, doing the 
translations of $K$ along $\wt \gamma$ we find a first interior point of contact
between $M$ and a translated copy  of $K$. Hence, $M=C$ by the maximum
principle, which leads to a contradiction.  This completes the proof.
\end{proof}

\smallskip

Now we state a generalization of the Asymptotic Theorem proved in 
\cite[Theorem 2.1]{ST2}. 

 Our result establishes some  obstruction for the asymptotic boundary of a complete
properly immersed minimal hypersurface in 
$\h^n\times\r.$

\begin{theo}[Asymptotic Theorem]\label{T.Asymptotic Theorem}
Let $\Gamma \subset \pain \hi n  \times \R$ be a connected 
$(n-1)$-submanifold with boundary. Let 
{ \em Pr} $ :   \pain \hi n  \times \R \rightarrow \pain \hi n$ be the
projection on the first factor.  Assume that: 
\begin{enumerate}
\item There is some point $q_\infty \in \partial\,${\em Pr}$( \Gamma)$
such that 
$q_\infty \not\in$ {\em Pr}$(\partial  \Gamma)$.
\item  $\Gamma \subset \pain \hi n \times (t_0,t_0 + 
\frac{\pi}{n-1})$ for some real number $t_0$.
\end{enumerate}
Then, there is no properly and completely immersed minimal hypersurface (maybe with finite
boundary) $M\subset \hi n \times \R$ such that 
$ \pain M =\Gamma$ and $M\cup \Gamma$ is a continuous
$n$-manifold with boundary.
\end{theo}

\begin{proof}
Assume, by contradiction, that there is such a minimal
hypersurface  $M$. Since $q_\infty \in\partial${Pr}$(\Gamma)$ and 
$q_\infty \not\in$ {Pr}$(\partial  \Gamma)$, there exists  a
$(n-1)$-geodesic plane 
$\omega\subset \hi n\times \{0\} $ such that a component $\omega^+$  of 
$\hi n \times \{0\} \setminus \omega$ satisfies:
\begin{enumerate}
\item $q_\infty \in \pain \omega^+$, $q_\infty\not\in \pain\omega$  
and   $\pain \omega^+ \cap${ Pr}$(\partial \Gamma)=\emptyset$.

\item If $M_0$ denotes the component of  $M \cap (\omega^+ \times \R)$
containing $q_\infty$ in its asymptotic boundary, then 
\begin{enumerate}
 \item $M_0\subset \hi n \times
(t_0,t_0 + \frac{\pi}{n-1})$ for some real number $t_0$.
\item $\partial  M_0 \subset \omega  \times 
 (t_0 {+ 2\varepsilon},t_0 {-2\varepsilon} +
\frac{\pi}{n-1})$ {for some $\varepsilon >0$.}

\end{enumerate}
\end{enumerate} 
Again, since $q_\infty \in\partial${Pr}$(\Gamma)$ and 
$q_\infty \not\in$ { Pr}$(\partial  \Gamma)$, there exists  a
$(n-1)$-geodesic plane 
$\pi\subset \hi n\times \{0\} $ such that a component $\pi^+$  of 
$\hi n \times \{0\} \setminus \pi$ satisfies:
\begin{enumerate}
\item $\pi^+ \subset \omega^+$.
\item  $\pain \pi^+ \cap\,${ Pr}$(\Gamma)=\emptyset$.
\item  $M_0 \cap (\pi^+ \times \R) =\emptyset$.
\end{enumerate}
Therefore we can find a compact part $K$ of a $n$-catenoid satisfying:
\begin{enumerate}
\item $K$ is connected.
\item $K \subset  \pi^+ \times (t_0+\varepsilon,t_0-\varepsilon + \frac{\pi}{n-1})$.
\item $\partial  K \subset \hi n \times 
\{t_0 {+\varepsilon}, t_0{-\varepsilon}  +
\frac{\pi}{n-1}\}$.
\end{enumerate}
We deduce consequently that $M_0 \cap K=\emptyset$. Then, considering the
horizontal translated copies of $K$ and arguing as in the proof of 
Theorem \ref{T.slab.catenoid}, we get a contradiction with the maximum
principle, which concludes the proof. 
\end{proof}

The following result is an immediate consequence of Theorem \ref{T.Asymptotic Theorem}.

\begin{coro}\label{C.asymptotic}
  Let $S_\infty \subset \pain \hi n \times \R$ be a
$(n-1)$-closed  continuous submanifold.
 Considering the halfspace model for $\hi n$, we can
assume that $S_\infty \subset \R^{n-1}\times\r  $.

  If $S_\infty$ is strictly convex in Euclidean  sense,
then there is no complete connected  properly
immersed minimal
hypersurface 
$M$ in $\hi n \times \R$, possibly with finite boundary, with asymptotic boundary
$S_\infty$ and such that 
$M\cup S_\infty$ is a continuous
$n$-manifold with boundary.
\end{coro}

\begin{remark}\label{R.graph}
{\em It follows from Corollary \ref{C.asymptotic} that there is no horizontal minimal
graph in $\h^n\times \r$, \cite[Equation (3)]{Sa2}, given by a positive function 
$g\in C^2(\Omega)\cap C^0(\ov \Omega)$, where 
$\Omega\subset\R^{n-1}\times\r   \subset \pain \hi n\times \R$ is a bounded strictly
convex domain in Euclidean sense,  assuming zero value  on $\partial \Omega$.}
\end{remark}

\end{document}